\documentclass[11pt]{amsart}
\makeindex

\usepackage[all, cmtip]{xy}
\usepackage{amssymb}
\usepackage{amsmath}
\usepackage{amsthm}
\usepackage{amscd}
\usepackage{amsfonts}
\usepackage{amssymb}
\usepackage[pdftex]{graphicx}
\usepackage{multirow}
\usepackage[usenames,dvipsnames,svgnames,table]{xcolor}
\usepackage{graphicx}
\usepackage{enumerate}

\usepackage[bookmarks=true,bookmarksnumbered=true,
 pdftitle={},
 pdfsubject={},
 pdfauthor={},
 pdfkeywords={TeX; dvipdfmx; hyperref; color;},
 colorlinks=true, linkcolor=BrickRed, citecolor=BlueGreen]
 {hyperref}

\theoremstyle{definition}

\theoremstyle{remark}

\theoremstyle{corollary}

\theoremstyle{theorem}

\theoremstyle{corollary}

\newtheorem{theorem}{Theorem}[section]
\newtheorem{lemma}[theorem]{Lemma}
\newtheorem{proposition}[theorem]{Proposition}

\theoremstyle{corollary}

\theoremstyle{definition}
\newtheorem{definition}[theorem]{Definition}

\theoremstyle{remark}
\newtheorem{remark}[theorem]{Remark}

\numberwithin{equation}{section}

\newcommand{\C}{\mathbb{C}}

\newcommand{\Q}{\mathbb{Q}}
\newcommand{\sO}{\mathcal{O}}
\def\P{\mathbb{P}}
\def\F{\mathbb{F}}
\def\G{\mathbb{G}}

\def\Fitt{\operatorname{Fitt}}

\def\reg{\operatorname{reg}}

\def\codim{\operatorname{codim}}

\def\embdim{\operatorname{embdim}}
\def\mult{\operatorname{mult}}
\def\Sing{\operatorname{sing}}
\def\Sec{\operatorname{Sec}}

\def\length{\operatorname{length}}

\DeclareMathOperator{\Bl}{Bl}                    

\newcommand{\sP}{\mathcal{P}}

\newcommand{\suchthat}{\;\ifnum\currentgrouptype=16 \middle\fi|\;}

\input xy
\xyoption{all}

\title[A Castelnuovo-Mumford regularity bound for threefolds]{A Castelnuovo-Mumford regularity bound for \\ threefolds with rational singularities}

\begin{document}

\author{Wenbo Niu}
\address{Department of Mathematical Sciences, University of Arkansas, Fayetteville, AR, 72701, USA}
\email{wenboniu@uark.edu}

\author{Jinhyung Park}
\address{Department of Mathematical Sciences, KAIST, Daejeon 34141, Republic of Korea}
\email{parkjh13@kaist.ac.kr}

\thanks{J. Park was supported by the National Research Foundation (NRF) funded by the
Korea government (MSIT) (NRF-2021R1C1C1005479).}

\subjclass[2010]{14N05, 14B05, 13D02}

\date{\today}

\dedicatory{Dedicated to Professor Sijong Kwak on the occasion of his sixtieth birthday.}

\keywords{Castelnuovo-Mumford regularity, projection, threefold, rational singularity, Loewy length, secant variety}

\begin{abstract}
The purpose of this paper is to establish a Castelnuovo-Mumford regularity bound for threefolds with mild singularities. Let $X$ be a non-degenerate normal projective threefold in $\P^r$ of degree $d$ and codimension $e$. We prove that if $X$ has rational singularities, then $\reg(X) \leq d-e+2$. Our bound is very close to a sharp bound conjectured by Eisenbud-Goto. When $e=2$ and $X$ has Cohen-Macaulay Du Bois singularities, we obtain the conjectured bound $\reg(X) \leq d-1$, and we also classify the extremal cases.
To achieve these results, we bound the regularity of fibers of a generic projection of $X$ by using Loewy length, and also bound the dimension of the varieties swept out by secant lines through the singular locus of $X$.
\end{abstract}

\maketitle


\section{Introduction}

\noindent Throughout the paper, we work over the field $\C$ of complex numbers. 
Let $X \subseteq \P^r$ be a non-degenerate projective variety of degree $d$ and codimension $e$.
We say that $X \subseteq \P^r$ is \emph{$m$-regular} in the sense of Castelnuovo-Mumford if $H^i(\P^r, \mathcal{I}_{X|\P^r}(m-i))=0$ for all $i>0$. It is equivalent to saying that $X \subseteq \P^r$ is $(m-1)$-normal and $\mathcal{O}_X$ is $(m-1)$-regular, i.e., the natural restriction map $H^0(\P^r, \mathcal{O}_{\P^r}(m-1)) \to H^0(X, \mathcal{O}_X(m-1))$ is surjective and $H^i(X, \mathcal{O}_X(m-1-i))=0$ for $i>0$. The minimal such number $m$ is called the \emph{Castelnuovo-Mumford regularity} of $X$ and is denoted by $\reg(X)$. For more details, we refer to \cite[Section 1.8]{positivityI}.

One of the fundamental problems in projective algebraic geometry is to find a bound for $\reg(X)$ in terms of algebraic and geometric data such as the degree $d$ and the codimension $e$. A sharp bound was suggested by Eisenbud-Goto \cite{EG} (cf. \cite{GLP}) in the \emph{regularity conjecture}, which says that
$$
\reg(X) \leq d-e+1.
$$
There is also a folklore conjecture which predicts that the maximal degree, denoted by $\max\deg(X)$, of minimal generators of the defining ideal $I_{X|\P^r}$ of $X \subseteq \P^r$ is at most $d$. Since $\max\deg(X) \leq \reg(X)$, the regularity conjecture implies this folklore conjecture. The regularity conjecture has only been proved for a few cases: integral curves \cite{C}, \cite{GLP}, connected curves \cite{Giaimo:CMRegofCurves}, smooth surfaces \cite{P}, \cite{L}, singular surfaces with rational or Gorenstein elliptic singularities \cite{Ni}, and smooth threefold in $\P^r$ with $r\geq 9$ \cite{R1}. Slightly weaker bounds for lower dimensional smooth varieties were obtained in \cite{Kw1, Kw2}.  A sharp bound for $\reg(\mathcal{O}_X)$ was shown in \cite{KP}, \cite{No} for a smooth variety $X$ of arbitrary dimension. This result is generalized in \cite{MPS} to a variety $X$ with $\Q$-Gorenstein isolated singularities.

It turns out that the regularity conjecture is not true. McCullough-Peeva \cite{MP} constructed counterexamples to the regularity conjecture, but their examples are highly singular. It is then a natural problem to make a distinction between mildly singular varieties satisfying the regularity conjecture and highly singular varieties not satisfying the regularity conjecture. The threefold case is particularly interesting because one of the counterexamples to the regularity conjecture is a projective threefold $X$ in $\P^5$ such that $\dim X_{\Sing}=1, \deg(X)=31,$ and $\max\deg(X)=\reg(X)=38$  (\cite[Example 4.7]{MP}) while the regularity conjecture holds for a smooth projective threefold in $\P^5$ by \cite{Kw}. Notice also that the singular threefold $X$ in $\P^5$ does not even fulfill the folklore conjecture aforementioned. 

In this paper, we consider normal projective threefolds with rational or Du Bois singularities. Those singularities are naturally appearing inevitably in many places such as minimal model program, moduli theory, and Hodge theory. Recently, it was shown that the secant varieties of smooth projective curves have normal  Du Bois singularities (see \cite{ENP}). In a variety of flavor, it is highly desirable to have the conjectured bound or a slightly weaker bound for $\reg(X)$, where $X$ is a normal projective threefold with rational or Du Bois singularities. It is also interesting to see how local geometric properties of those singularities are related to global algebraic properties of the embedded projective varieties.

The first main result of this paper is to establish the following regularity bound for threefolds with rational singularities, extending Kwak's result for smooth threefolds in \cite{Kw1}. Our bound is only off by 1 from the bound in the regularity conjecture, and in particular, the maximal degree can be bounded by the degree.

\begin{theorem}\label{main}
Let $X \subseteq \P^r$ be a non-degenerate normal projective threefold of degree $d$ and codimension $e$. Suppose that $X$ has rational singularities. Then we have
$$
\reg(X) \leq d-e+2.
$$
\end{theorem}

Suppose furthermore that $X$ has only isolated rational singularities and
\begin{equation}\tag{$\star$}\label{techassump}
\dim \bigcup_{l \in \Sigma} l \leq 4, \text{ where } \Sigma:=\{\text{lines } l \subseteq \P^r \mid \length(l \cap X) \geq 4 \text{ and } l \cap X_{\Sing} = \emptyset  \}.
\end{equation}
In this case, our approach yields the optimal result
$$
\reg(X) \leq d-e+1.
$$
See Remark \ref{rem:furthercase} for more details.
Note that Ran \cite[Theorem 1]{R1}  proved that if $r \geq 9$ and $X$ has non-degenerate tangent variety, then the condition (\ref{techassump}) holds. As remarked in \cite[Remark 2.1]{R1}, threefolds with degenerate tangent variety is classified in \cite[(5.48)]{GH}.

The codimension two case is specially important as one of the counterexamples to regularity conjecture by McCullough-Peeva appeared in this case. The second main result of this paper is to establish a sharp regularity bound for codimension two normal projective threefolds with Cohen-Macaulay Du Bois singularities and to classify the extremal cases. This generalizes the main theorem of \cite{Kw}, which deals with the smooth case. Note that rational singularities are Cohen-Macaulay Du Bois singularities.

\begin{theorem}\label{main2}
Let $X \subseteq \P^5$ be a non-degenerate normal projective threefold of degree $d$. Suppose that $X$ has Cohen-Macaulay Du Bois singularities. Then we have
$$
\reg(X) \leq d-1.
$$
Furthermore, $\reg(X)=d-1$ if and only if $X \subseteq \P^5$ is either a variety of minimal degree or a complete intersection of type $(2,2)$.
\end{theorem}

If $X \subseteq \P^r$ is a non-degenerate normal projective threefolds with Cohen-Macaulay Du Bois singularities of degree $d$ and codimension $e$, then we have
$$
\reg(X) \leq d-e+3 + {e+1 \choose 3}.
$$
See Remark \ref{rem:furthercase} for more details.

To prove Theorems \ref{main} and \ref{main2}, we take a \emph{general projection} 
$$
\pi \colon X \to \overline{X} \subseteq \P^4
$$
of a non-degenerate normal projective threefold $X$ in $\P^r$ to a hypersurface $\overline{X}$ in $\P^4$, and analyze the complexity of the fibers of $\pi$. It is known that if one has a smaller Castelnuovo-Mumford regularity bound for all fibers of $\pi$, then one would get a better Castelnuovo-Mumford regularity bound for $X$. See Section \ref{proofsec} for more details. This idea essentially goes back to Castelnuovo \cite{C} (see also  \cite{P}), and is combined with cohomological method by Lazarsfeld \cite{L} to show a sharp regularity bound for smooth surfaces. This technique has been successfully applied to lower dimensional smooth varieties as in \cite{Kw1, Kw, Kw2}, \cite{R} and singular surfaces as in \cite{Ni}.

Along this line, our task is to find a reasonably small uniform regularity bound for every fiber $X_y=\pi^{-1}(y)$ of general projection $\pi  \colon X \to \overline{X} \subseteq \P^4$ with $y \in \overline{X}$ (Theorem \ref{regfiber}). For this purpose, we essentially classify the fibers of $\pi$ (see Remark \ref{rmk}). The situation is well understood if $X$ is nonsingular; every $X_y$ has length at most $4$, and there are only finitely many $y$ with $X_y$ collinear of length $4$ (see \cite{Kw1}).  
However, if $X$ is singular, then the situation becomes more complicated and new ideas are essentially needed. When $X_y$ is supported in the smooth locus of $X$, we can control the length of $X_y$ and consequently the regularity of $X_y$ by dimension counting arguments and the (dimension+2)-secant lemma \cite[Theorem 1]{R} (Proposition \ref{mather}). When $X_y$ contains a singular point of $X$, the situation becomes quite delicate. It is worth pointing out that the length of $X_y$ is not useful anymore to bound its regularity at least in our case. To remedy this situation, we use \emph{Loewy length} instead to bound the regularity of fibers. 
Since $X$ has rational or Cohen-Macaulay Du Bois singularities, we can apply \emph{Brian\c{c}on-Skoda type theorem} to bound the Loewy length (see Lemma \ref{red=>loewy} and Proposition \ref{regfibsing3}). This is enough to establish the regularity bound for $X$ if $X$ has only isolated singularities.
However, a normal projective threefold with rational or Cohen-Macaulay Du Bois singularities could have positive dimensional singular locus, and this is actually the major obstruction of bounding regularity of $X$.
The main novelty of this paper is to overcome the difficulty by computing the dimension of the \emph{secant variety} swept out by secant lines passing through the singular locus of $X$ (see Theorem \ref{p:01}). This implies that the support of $X_y$ has at most two points (Proposition \ref{fibercurve}), and this result leads us to the desired control on the regularity of the fiber $X_y$. See Section \ref{fibersec} for more details.

The following theorem describes the secant variety swept out by secant lines passing through the singular locus for an arbitrary projective variety.  It might be of independent interest, and can be viewed as complementary to the results on secant variety of smooth locus (see \cite[Theorem 1]{R}, \cite[Theorem A.3]{BDELU}). We expect that it would have other applications to general projections of higher dimensional singular varieties.
 
\begin{theorem}\label{p:01}
Let $X \subseteq \P^r$ be a projective variety of dimension $n$, and  let $Z$ be a nonempty proper closed subset containing $X_{\Sing}$. Denote $\delta:=\codim_{X}Z$ and $U_Z:=X-Z$. Consider the set of secant lines
$$
\Sigma:=\{\text{lines } l \subseteq \P^r \mid \length(l \cap U_Z) \geq n+1-\delta \text{ and } l \cap Z \neq \emptyset  \}.
$$
Then we have
 $$\dim \overline{S_{\Sigma}(X)}\leq n+1,\text{ where }S_{\Sigma}(X):=\bigcup_{l \in \Sigma} l.$$
\end{theorem}

It is worth noting that Theorems \ref{main} and \ref{main2} show a distinction between normal projective threefolds with rational or Cohen-Macaulay Du Bois singularities and highly singular projective threefolds. 
There are two geometric reasons for this phenomenon. First, normal projective threefolds with rational or Cohen-Macaulay Du Bois singularities satisfy \emph{Kodaira vanishing theorem}, so the cohomological method can be applied. Second, all fibers of generic projections of these varieties have reasonably small Loewy length, so the general projection method works. There do exist normal projective threefolds with rational or Cohen-Macaulay Du Bois singularities having one dimensional singular locus. The first secant variety $\Sigma_1=\Sigma_1(C, L)$ of a smooth projective curve $C$ of genus $g$ embedded by the complete linear system of a line bundle $L$ of degree $\geq 2g+3$ has rational singularities when $g=0$ and has normal Cohen-Macaulay Du Bois singularities when $g \geq 1$, and the singular locus of $\Sigma_{1}$ is $C$ (see \cite{ENP}). If $(g, \deg L) = (0, 5), (1,6), (2,7)$, then $\Sigma_1$ has codimension two. In this case, $\deg(\Sigma_1) = 6,9,13$ and $\reg(\Sigma_1) = 3, 5, 5$, respectively (cf. counterexample of threefold in codimension two \cite[Example 4.7]{MP}).

 \medskip

The  paper is organized as follows. We begin in Section \ref{secantsec} with the proof of Theorem \ref{p:01}. In Section \ref{singsec}, we study local properties of mild singularities using Loewy length.
Section \ref{fibersec} is devoted to the study of complexity of fibers of a generic projection. Finally, in Section \ref{proofsec}, by the classical generic projection method, we complete the proofs of Theorems \ref{main} and \ref{main2}.

\medskip

\noindent {\em Acknowledgment}. We would like to thank Lawrence Ein, Sijong Kwak and Lance Edward Miller for the inspiring suggestions and comments. We are grateful to the referee for helpful comments.

\section{Dimension of Secant Variety}\label{secantsec}

\noindent We give the proof of Theorem \ref{p:01} in this section. The theorem plays a crucial role in controlling regularity of fibers of a general projection passing through singular locus. 

\begin{proof} [Proof of Theorem \ref{p:01}]
We prove the theorem by contradiction.  Assume otherwise that $\dim \overline{S_{\Sigma}(X)}\geq n+2$. Denote by $\G=\G(\P^1, \P^r)$ the Grassmannian variety parameterizing lines in $\P^r$. We can find a subvariety $B_0 \subseteq \G$ of dimension $n+1$ parameterizing lines $l$ in $\Sigma$.  We  then obtain a diagram:
\begin{equation}\label{diagram}
	\xymatrix{
		X \subseteq \P^r& \Lambda_B \ar[l]_-{\pi} \ar[d]_{\phi}^{\P^1\text{-bundle}} & \\
		& B \ar[r] & B_0 \subseteq \G,
	}
\end{equation}
	where $B$ is a smooth projective variety of dimension $n+1$ mapping birationally to $B_0$, the morphism $\phi \colon \Lambda_B \to B$ is the pull-back to $B$ of the tautological $\P^1$-bundle on $\G$, and $\pi$ is the natural projection.
 Write
	$$I_{X|\P^r}\cdot \sO_{\Lambda_B}=\sO_{\Lambda_B}(-\sum a_iE_i-F)\otimes I_{W|\Lambda_B}$$
	where the ideal $I_{W|\Lambda_B}$ defines a subscheme $W$ with $\codim_{\Lambda_B} W\geq 2$ and $E_i$ and $F$ are divisors such that $E_i$ dominates $B$ but $F$ does not. Let
	$$Y_i:=C_X(E_i)=\pi(E_i), \ \delta_i:=\codim_X Y_i,\text{ and }e:=\codim_{\P^r}X.$$
	Note that the divisors $E_i$ can be  divided into two disjoint sets based on whether their centers are contained in $Z$ or not; the first set includes all $E_i$ with $C_X(E_i)\nsubseteq Z$ while the second one includes all $E_i$ with $C_X(E_i)\subseteq Z$. The latter one is not empty because $I_{X|\P^r} \subseteq I_{Z|\P^r}$ and all secant lines in $B$ touch $Z$, which in other words means that the subscheme defined by $I_{Z|\P^r} \cdot \sO_{\Lambda_B}$ dominates $B$. 
	
	We shall do the local calculation for each $E_i$. To be more concrete, let us consider the divisor $E_1$. Choose a general point $y\in E_1$ and set $x=\pi(y)$. We may assume that $x$ is a nonsingular point of $Y_1$. Hence we can choose a local system of parameters of $\sO_{\P^r,x}$ as
	$$\underbrace{x_1,\ldots,  x_e,x_{e+1},\ldots, x_{e+\delta_1}}_{\text{ local equations for }Y_1}, x_{e+\delta_1 + 1},\ldots, x_r$$
	and a local system of parameters of $\sO_{\Lambda_B,y}$ as
	$$y_1,\ldots, y_{n+2}$$
	such that $E_1$ is defined by $y_1$. Then, by assumption, one can write
	$$x_1=y^{b_1}_1w_1, \ldots, x_e=y^{b_1}_1w_e,\ x_{e+1}=y_1w_{e+1}, \ldots, x_{e+\delta_1}=y_1w_{e+\delta_1},$$
where $w_i$ are elements in the local ring $\sO_{\Lambda_B,y}$ and the value of $b_1$ is assigned as follows: if $C_X(E_1)\nsubseteq Z$, we can further assume that $x$ is a nonsingular point of $X$ so that $x_1,\ldots, x_e$ are the local equations for $X$, and we set $b_1=a_1$; if $C_X(E_1)\subseteq Z$, we simply set $b_1=1$.
	
From the morphism $\pi \colon \Lambda_{B} \to \P^r$ in the diagram (\ref{diagram}), we have an exact sequence
	\begin{equation}\label{eq:01}
	\pi^*\Omega_{\P^r}^1\longrightarrow \Omega_{\Lambda_B}^1\longrightarrow \Omega_{\Lambda_B/\P^r}^1\longrightarrow 0.
	\end{equation}
	We consider the Fitting ideal $\Fitt^0(\Omega_{\Lambda_B/\P^r}^1)$.
	Localizing the sequence (\ref{eq:01}) at $y$, one has
	$$\Omega_{\P^r,x}^1\stackrel{\varphi}{\longrightarrow}\Omega_{\Lambda_B,y}^1\longrightarrow \Omega_{\Lambda_B/\P^r,y}^1\longrightarrow 0.$$
	Choose the basis ${dx_i}$ for $\Omega_{\P^r,x}^1$ and the basis $dy_j$ for $\Omega_{\Lambda_B,y}^1$. Then the map $\varphi $ can be represented as an $(n+2)\times r$ matrix
	$$\varphi=
	\left( \begin{array}{cccc}
	\frac{\partial x_1}{\partial y_1} & \frac{\partial x_2}{\partial y_1} & \cdots &  \frac{\partial x_r}{\partial y_1} \\
	\frac{\partial x_1}{\partial y_2} & \frac{\partial x_2}{\partial y_2} & \cdots &  \frac{\partial x_r}{\partial y_2} \\
	\vdots & \vdots & \ddots & \vdots \\
	\frac{\partial x_1}{\partial y_{n+2}} & \frac{\partial x_2}{\partial y_{n+2}} &\cdots &  \frac{\partial x_r}{\partial y_{n+2}}
	\end{array} \right).
	$$
	We then see that the $0$-th Fitting ideal of $\Omega_{\Lambda_B/\P^r,y}^1$, which is generated by the $(n+2)\times (n+2)$ minors of $\varphi$,  has the form 	
	$$\Fitt^0(\Omega_{\Lambda_B/\P^r,y}^1)=y^{2b_1-1+\delta_1}_1\cdot J$$
	for some ideal $J$. Hence we conclude that $\Fitt^0(\Omega_{\Lambda_B/\P^r,y}^1)\subseteq  (y^{2b_1-1+\delta_1}_1)$.
	This argument works for all $E_i$, and therefore globally we have
	$$\Fitt^0(\Omega_{\Lambda_B/\P^r}^1)\subseteq \sO_{\Lambda_B}(-\sum (2b_i-1+\delta_i)E_i).$$
	Next we choose a general secant line $l\in B$. So the above inclusion shows that
	$$\Fitt^0(\Omega_{\Lambda_B/\P^r}^1)\cdot \sO_l\subseteq \sO_{\Lambda_B}(-\sum (2b_i-1+\delta_i)E_i)\cdot \sO_l.$$
	Let $D$ be the divisor on $l$ defined by $\Fitt^0(\Omega_{\Lambda_B/\P^r}^1)\cdot \sO_l$. Then the above inclusion shows that
	\begin{eqnarray*}
		\deg(D)& \geq & (\sum (2b_i-1+\delta_i)E_i)\cdot l\nonumber\\
		&=& \sum_{C_X(E_i)\nsubseteq Z} (2a_i-1+\delta_i)E_i\cdot l +\sum_{C_X(E_i)\subseteq Z} (1+\delta_i)E_i\cdot l \nonumber \\
		&=& \sum_{C_X(E_i)\nsubseteq Z} a_iE_i\cdot l +\sum_{C_X(E_i)\nsubseteq Z} (a_i-1+\delta_i)E_i\cdot l+\sum_{C_X(E_i)\subseteq Z} (1+\delta_i)E_i\cdot l \nonumber \\
		&\geq & \text{length}(l\cap U_Z)+\sum_{C_X(E_i)\subseteq Z} (1+\delta_i)E_i\cdot l \nonumber\\
		& \geq & (n+1-\delta) + (1+\delta) ~\,~\,~\, \text{ ~(recall: $\{E_i \mid C_X(E_i) \subseteq Z\} \neq \emptyset$)} \nonumber\\
		& = &  n+2. \nonumber
	\end{eqnarray*}
On the other hand, write $T$ as the tangent space of $B$ at $l$. There is a commutative diagram
	$$\begin{CD}
	0@>>>N^*_{l/\P^r} @>>> \Omega_{\P^r}^1|_l@>>> \Omega^1_l @>>>0\\
	@.@VVV @VVV @|\\
	0@>>>T^*\otimes \sO_l @>>> \Omega^1_{\Lambda_B}|_l@>>> \Omega^1_l @>>> 0.
	\end{CD}$$\\
	Note that $N^*_{l/\P^r}=A\otimes \sO_l(-1)$ where $A$ is an $(r-1)$-dimensional vector space. By the Snake Lemma, we obtain an  exact sequence
	$$A\otimes \sO_l(-1)\longrightarrow T^*\otimes \sO_l\longrightarrow \Omega_{\Lambda_B/\P^r}^1|_l\longrightarrow 0.$$
	The Eagon-Northcott complex gives an exact sequence
	$$\wedge^{n+1}A\otimes \sO_{l} (-(n+1))\longrightarrow \det T^*\otimes \sO_l\longrightarrow \sO_D\longrightarrow 0.$$
	The semistability of the vector bundle $\wedge^{n+1}A\otimes \sO_l (-(n+1))$ implies that $\deg (D)\leq n+1$. Hence combing all together, we have
	$$n+1\geq \deg(D)\geq n+2, $$
	which gives the desired contradiction.	
\end{proof}

\begin{remark} 
The proof is inspired by \cite[Theorem A.3]{BDELU}. However, instead of applying the generic projection argument as in \cite{BDELU}, our proof uses the Fitting ideal. 
In addition, Theorem \ref{p:01} does allow the case that $X_{\Sing}$ is empty, i.e., $X$ is nonsingular. In this case, $S_{\Sigma}$ is swept out by the secant lines passing through the nonempty proper closed subset $Z$.
\end{remark}

\section{Loewy Length and Mild Singularities}\label{singsec}

\noindent In this section, we study local properties of mild singularities. Especially, we bound the regularity of the intersection of the singularities with general linear spaces by using Loewy length (see Proposition \ref{regfibsing3}).

\begin{definition}
For an Artinian local ring $(A, \mathfrak{m})$, we define the \emph{Loewy length} $ll(A)$ of $A$ to be the nonnegative number $ll(A) := \max \{ i \mid \mathfrak{m}^i \neq 0 \}$. If $\mathfrak{m}=0$, i.e., $A$ is a field, then we put $ll(A)=0$.
\end{definition}

The Loewy length of an Artinian local ring is usually much smaller than its length. It provides an efficient way to bound the regularity of  zero dimensional schemes, as shown in the following.

\begin{theorem}[{\cite[Theorem 2.2]{Ni}}]\label{loewy-reg}
Let $X \subseteq \P^r$ be a zero dimensional subscheme supported at distinct closed points $p_1, \ldots, p_t$. For each $1 \leq i \leq t$, set $\mu_i:=ll({\mathcal{O}_{X, p_i}})$ to be the Loewy length of the local ring $\mathcal{O}_{X, p_i}$. Then $X$ is $(\mu_1 + \cdots  +\mu_t + t)$-regular.
\end{theorem}

To bound the Loewy length, we need to consider a reduction of an ideal.

\begin{definition}
Let $(R, \mathfrak{m})$ be a local Noetherian ring. An ideal $J \subseteq \mathfrak{m}$ is called a \emph{reduction} of $\mathfrak{m}$ if $\mathfrak{m}^{k+1} = J \mathfrak{m}^k$ for some integer $k \geq 0$. Moreover, $J$ is called a {\em minimal reduction} if it is a reduction minimal with respect to inclusion.
\end{definition}

As the setting in the definition, if $J$ is a reduction of $\mathfrak{m}$, then $J$ always contains a minimal reduction (\cite[Theorem 8.3.5]{HS}). Furthermore, if $\dim R = n$ and $R/\mathfrak{m}=\C$, then there exists a nonempty Zariski open subset $U$ in the $n$-th Cartesian product $(\mathfrak{m}/\mathfrak{m}^2)^n$ of the cotangent space such that if $x_1, \ldots, x_n \in \mathfrak{m}$ with $(x_1 + \mathfrak{m}^2, \ldots, x_n + \mathfrak{m}^2) \in U$, then $(x_1, \ldots, x_n)$ is a reduction of $\mathfrak{m}$ (\cite[Theorem 8.6.6]{HS}).

\medskip

Turning to geometric setting, let $X \subseteq \P^r$ be a projective variety of dimension $n$, and $p \in X$ be a point. Let $L\subseteq\P^r$ be an $(r-k)$-dimensional linear subspace passing through $p$, and assume that $L$ is cut out by linear forms $l_1,\ldots, l_k$ on $\P^r$. Locally at the point $p$, each form $l_i$ gives an element $\bar{l}_i$ of $\mathfrak{m}_{X, p}$ via the quotient $\mathfrak{m}_{X,p}=\mathfrak{m}_{\P^r,p}/I_{X,p}$. Thus we obtain an ideal $(\bar{\l}_1, \ldots, \bar{l}_k) \subseteq \mathfrak{m}_{X, p}$ generated by the elements $\bar{l}_i$. This ideal is the localization of the ideal sheaf $\mathcal{I}_L\cdot \sO_X$ at $p$. We say that $L$ is a \emph{reduction linear subspace at $(X,p)$} or simply $L$ is a \emph{reduction at $(X,p)$} if the ideal $(\bar{l}_1, \ldots, \bar{l}_k)$ is a reduction of $\mathfrak{m}_{X, p}$. 
If $L$ is reduction at $(X,p)$, then the intersection $X\cap L$ has dimension zero at $p$ and $k\geq n$. For two linear subspaces $L, L'$ with $p \in L' \subseteq L$, if $L$ is a reduction at $(X,p)$, then so is $L'$.

\begin{lemma}\label{red-loewy}
Let $X \subseteq \P^r$ be a projective variety of dimension $n$, and $p \in X$ be a point. Suppose that $m$ is a positive integer such that  $\mathfrak{m}_{X, p}^m \subseteq J$ for any minimal reduction $J$ of $\mathfrak{m}_{X, p}$. If $L$ is a reduction linear subspace at $(X,p)$ in $\P^r$, then we have $ll({\mathcal{O}_{X \cap L, p}}) \leq m-1$.
\end{lemma}

\begin{proof}
Let $l_1, \ldots, l_k$ be linear forms on $\P^r$ cutting out $L$. Locally at the point $p$, these linear forms generate an ideal  $(\bar{\l}_1, \ldots, \bar{l}_k)\subseteq \mathfrak{m}_{X, p}$  which is a reduction ideal of $\mathfrak{m}_{X, p}$, and hence, $\mathfrak{m}_{X, p}^m \subseteq (\bar{l}_1, \ldots, \bar{l}_k)$ since the latter contains a minimal reduction. As $\mathcal{O}_{X \cap L, p} = \mathcal{O}_{X, p}/(\bar{l}_1, \ldots, \bar{l}_k)$, we see immediately that $\mathfrak{m}_{X \cap L, p}^m=0$ as desired. 
\end{proof}

To apply the above lemma, we need to bound the number $m$. This number turns out to depend on the singularities of the local ring. Let us recall definitions of some mild singularities that provide us reasonable bounds on $m$.

\begin{definition}
Let $X$ be a projective variety. 
\begin{enumerate}[\indent$(1)$]
	\item We say that $X$ has \emph{rational singularities} if $X$ is normal and 
there exists a proper birational morphism $f \colon Y \to X$ from a smooth variety $Y$ such that $R^i f_* \mathcal{O}_Y=0$ for $i>0$. Note that $R^i f_* \mathcal{O}_Y=0$ for $i>0$ if and only if $X$ is locally Cohen-Macaulay and $f_* \omega_Y = \omega_X$ (\cite[Theorem 5.10]{KM}). 
	\item Let  $\underline{\Omega}_X^{\bullet}$ be the Deligne-Du Bois complex for $X$, which is a generalization of the de Rham complex for a nonsingular variety (see \cite[Chapter 6]{Ko2} for detail). We say that $X$ has \emph{Du Bois singularities} if the natural map
$$
\sO_X \longrightarrow \underline{\Omega}_X^{0}=Gr_{\text{filt}}^0 \underline{\Omega}_X^{\bullet}.
$$
is a quasi-isomorphism. 
\end{enumerate}
\end{definition}

Rational singularities contain a large range of important singularities considered in birational geometry such as terminal, canonical, or log terminal singularities (\cite[Theorem 5.22]{KM}). Du Bois singularities are another important singularities naturally appeared in birational geometry and Hodge theory.
Note that if $X$ has rational singularities or log canonical singularities, then $X$ has Du Bois singularities (\cite[Corollary 6.23 and Corollary 6.32]{Ko2}).
For further information about rational or Du Bois singularities and their relation with birational geometry, we refer to \cite{Ko2} and \cite{KM}.

\begin{lemma}\label{red=>loewy}
Let $X \subseteq \P^r$ be a normal projective threefold, $p \in X$ be a point, and $L$ be a reduction linear subspace at $(X, p)$ in $\P^r$. Then the following hold:
\begin{enumerate}[\indent$(1)$]
	\item If $X$ has rational singularities, then $ll({\mathcal{O}_{X \cap L, p}}) \leq 2$.
	\item If $X$ has Cohen-Macaulay Du Bois singularities, then $ll({\mathcal{O}_{X \cap L, p}}) \leq 3$.
\end{enumerate}
\end{lemma}

\begin{proof} Let $J$ be a minimal reduction of the maximal ideal $\mathfrak{m}_{X, p}$ of the local ring ${\mathcal{O}_{X,p}}$ at $p$. If $(X, p)$ is a $3$-dimensional rational singularity, then it follows from Brian\c{c}on-Skoda type theorem (\cite{LT}, see also \cite[Theorem 3.2 (1)]{HW}) that $\mathfrak{m}_{X, p}^3 \subseteq J$. If $(X, p)$ is a $3$-dimensional normal Cohen-Macaulay Du Bois singularity, then \cite[Lemma 3.5]{S} implies that $\mathfrak{m}_{X, p}^4 \subseteq J$. Thus the assertions follow from Lemma \ref{red-loewy}.
\end{proof}

\begin{remark} 
For a surface $S$ with rational singularities or Gorenstein elliptic singularities and for a closed point $p \in S$, it has been showed in \cite{Ni} that the inequality $\mult_p S \leq \embdim_p S$ plays a crucial role in bounding the Loewy length of fibers of a general projection of $S$. Ideally, following the method used there, if one could show  $\mult_p X \leq \embdim_p X -1$ for a threefold $X$,  then Lemma \ref{red=>loewy} (1)  can be established. However, this inequality is only known for terminal singularities by \cite[Theorem 2.1]{Ka} (see also \cite[Proposition 3.10]{TW}) and for Gorenstein canonical singularities by \cite[Theorem 5.1]{HW}. In general, if $(X, p)$ is a rational singularity, then we only have $\mult_p X \leq {{\embdim_p X -1}\choose{\dim X-1}}$ which is sharp (see \cite[Theorem 3.1 and Example 4.1]{HW}).
\end{remark}

\begin{remark}\label{rem-surDB}
Let $S \subseteq \P^r$ be a non-degenerate normal projective surface of degree $d$ and codimension $e$. Suppose that $S$ has Du Bois singularities.
As in the proof of Lemma \ref{red=>loewy}, by applying \cite[Lemma 3.5]{S}, we can show that the Loewy length of any fiber of a general projection $\pi \colon S \to \overline{S} \subseteq \P^3$ is at most 2. Then the arguments in \cite[Section 3]{Ni} show that the regularity conjecture holds for $S$, i.e.,
$$
\reg(S) \leq d-e +1.
$$
Note that all singularities considered in \cite{Ni} are Du Bois singularities.
\end{remark}

In the rest of the paper, we frequently use the notion of (partial) flag varieties. For given $m$ integers $1\leq k_1<k_2<\cdots <k_m\leq r-1$, the \emph{(partial) flag variety} $\F(\P^{k_1}, \ldots, \P^{k_m}, \P^r)$  parameterizes the $m$-tuples $(L_1,L_2,\ldots, L_m)$ of the nested linear subspaces $L_1\subseteq L_2\subseteq \cdots\subseteq L_m$ in $\P^r$ with $\dim L_i=k_i$. For a closed point $p\in \P^r$, we also consider the subvariety
$$
\F(p, \P^{k_1}, \ldots, \P^{k_m}, \P^r):=\{ (L_1, \ldots, L_m) \in \F(\P^{k_1}, \ldots, \P^{k_m}, \P^r) \mid p \in L_1 \}
$$
of the flag variety $\F(\P^{k_1}, \ldots, \P^{k_m}, \P^r)$. There is a canonical isomorphism 
$$
\F(p, \P^{k_1}, \ldots, \P^{k_m}, \P^r) \simeq \F(\P^{k_1-1}, \ldots, \P^{k_m-1}, \P^{r-1}).
$$ 
In the special case when $m=1$, the flag variety $\F(\P^{k_1},\P^r)$ is just the Grassmannian variety $\G(\P^{k_1},\P^r)$, and therefore, $\F(p, \P^{k_1}, \P^r) \simeq \G(\P^{k_1-1}, \P^{r-1})$. 

\medskip

Now we are ready to prove the following main result of this section. The proposition plays an important role in controlling regularity of fibers of a general projection passing through singular locus. The essential idea of the proof is to use the classical dimension counting method.

\begin{proposition}\label{regfibsing3}
Let $X \subseteq \P^r$ be a normal projective threefold, and $\Lambda \subseteq \P^r$ be a general $(r-5)$-dimensional linear subspace $\Lambda$ in $\P^r$ disjoint from $X$. Then for every singular point $p \in X_{\Sing}$, the linear subspace $L_p= \langle \Lambda,  p \rangle$ in $\P^r$ spanned by $\Lambda$ and $p$ is a reduction at $(X,p)$. In particular, the following hold:
\begin{enumerate}[\indent$(1)$]
	\item Suppose that $X$ has rational singularities. Then $ll({\mathcal{O}_{X \cap L_p, p}}) \leq 2$. If furthermore $X \cap L_p$ supports only at $p$, then $X \cap L_p \subseteq L_p \simeq \P^{r-4}$ is $3$-regular.
	\item Suppose that $X$ has Cohen-Macaulay Du Bois singularities. Then $ll({\mathcal{O}_{X \cap L_p, p}}) \leq 3$. If furthermore $X \cap L_p$ supports only at $p$, then $X \cap L_p \subseteq L_p \simeq \P^{r-4}$ is $4$-regular.
\end{enumerate}
\end{proposition}

\begin{proof}
Let $p \in X_{\Sing}$ be a singular point of $X$.
Let $\mathfrak{m}=\mathfrak{m}_{X, p}$ be the maximal ideal of the local ring $\mathcal{O}_{X,p}$, and $\mathfrak{m}_{\P^r,p}$ be the maximal ideal of the local ring $\mathcal{O}_{\P^r,p}$.
Consider 
$$
V_p := \overline{\{ L \in \F(p, \P^{r-3}, \P^r) \mid L \text{ is not a reduction at $(X,p)$} \}} \subseteq \F(p, \P^{r-3}, \P^r).
$$
As $\mathfrak{m}=\mathfrak{m}_{\P^r,p}/I_{X,p}$, it is generated by the quotients of linear forms. By \cite[Theorem 8.6.6]{HS}, we can choose three general linear forms $l_1,l_2,l_3$ so that the $3$-tuple of their quotients $(\bar{l}_1,\bar{l}_2,\bar{l}_3)$ in the product $(\mathfrak{m}/\mathfrak{m}^2)^3$  lies in the Zariski open subset stated in \cite[Theorem 8.6.6]{HS}. Hence the linear space defined by $l_1,l_2,l_3$ is a  reduction at $(X,p)$. This suggests that $V_p$ is a proper closed subset of $\F(p, \P^{r-3}, \P^r)$. Thus $\dim V_p \leq \dim \F(p, \P^{r-3}, \P^r)-1 = 3r-10$. Let
$$
V_p' := \overline{\{ L' \in \F(p, \P^{r-4}, \P^r) \mid L' \text{ is not a reduction at $(X, p)$} \}} \subseteq \F(p, \P^{r-4}, \P^r).
$$
We claim that 
$$
\text{$\dim V_p' \leq 4r-18$, i.e., $\codim_{\F(p, \P^{r-4}, \P^r)} V_p' \geq 2$.}
$$
To see the claim, we consider the diagram of flag varieties
\[
\begin{split}
\xymatrix{
& \F(p, \P^{r-4}, \P^{r-3}, \P^r) \ar[dl]_-f \ar[dr]^-g& \\
\F(p,  \P^{r-3}, \P^r)  & & \F(p, \P^{r-4}, \P^r),
}
\end{split}
\]
where $f$ and $g$ are natural projections. The dimensions of any fibers of $f$ and $g$ are $r-4$ and $3$, respectively. Let 
$$
W:=\{ L' \in \F(p, \P^{r-4}, \P^r) \mid g^{-1}(L') \subseteq f^{-1}(V_p) \} \subseteq \F(p, \P^{r-4}, \P^r).
$$ 
It is clear that we have the inclusion $V_p' \subseteq W$. Thus to prove the claim, it is enough to show that $W$ has codimension at least $2$ in $\F(p, \P^{r-4}, \P^r)$. Now, we take a prime divisor $H$ on $\F(p, \P^{r-3}, \P^r)$ containing an irreducible component of $V_p$. Then $f^{-1}(H)$ is a prime divisor on $\F(p, \P^{r-4}, \P^{r-3}, \P^r)$. We show that the restricted morphism 
$$
u:=g|_{f^{-1}(H)}: f^{-1}(H)\rightarrow \F(p, \P^{r-4}, \P^r)
$$
is surjective. Indeed, for any point $L'\in \F(p, \P^{r-4}, \P^r)$,  the set $f(g^{-1}(L'))$ parameterizes $(r-3)$-dimensional linear subspaces that contain the $(r-4)$-dimensional subspace represented by $L'$. Thus $\dim f(g^{-1}(L')) = 3$. However, since every effective divisor on $\F(p, \P^{r-3}, \P^r)$ is very ample, $f(g^{-1}(L'))$ meets $H$. This implies that $g^{-1}(L')$ meets $f^{-1}(H)$. Thus $u$ is surjective as desired. 
Note that $\dim f^{-1}(H)=(3r-10)+(r-4)=4r-14$. So a general fiber of  $u$ has dimension $2$. Furthermore, for any prime divisor $H'$ on $\F(p, \P^{r-4}, \P^r)$, it is impossible that $g^{-1}(H') \subseteq f^{-1}(H)$. This means that the locus of points in $\F(p, \P^{r-4}, \P^r)$ over which the fiber of $u$ has dimension $3$ has codimension at least two. Notice that $g^{-1}(W) \subseteq f^{-1}(H)$ and $\dim g^{-1}(L') = 3$ for all $L' \in W$. Hence $\codim_{\F(p, \P^{r-4}, \P^r)} W  \geq 2$, so the claim follows.

Now each $V'_p$ discussed above is naturally a subset of the Grassmannian $\G(\P^{r-4}, \P^r)$. Let 
$$
V:=\bigcup_{p \in X_{\Sing}} V_p' \subseteq \G(\P^{r-4}, \P^r).
$$ 
Since $\dim X_{\Sing} \leq 1$, it follows from the claim above that $\dim V \leq 4r-17$. Consider the diagram of the flag varieties
\[
\begin{split}
\xymatrix{
& \F(\P^{r-5}, \P^{r-4}, \P^r) \ar[dl]_-{f'} \ar[dr]^-{g'}& \\
\G(\P^{r-4}, \P^r) & & \G(\P^{r-5},  \P^r)
}
\end{split}
\]
with the natural projections $f'$ and $g'$. The dimensions of any fibers of $f'$ and $g'$ are $r-4$ and $4$, respectively. Then $\dim f'^{-1}(V) \leq  (4r-17)+(r-4)=5r-21$. Since $\dim \G(\P^{r-5}, \P^r)=5r-20$, we have $g'^{-1}(\Lambda) \cap f'^{-1}(V) = \emptyset$ for a general member $\Lambda \in \G(\P^{r-5}, \P^r)$. Therefore, for every singular point $p \in X_{\Sing}$, the linear subspace $L_p=\langle p, \Lambda \rangle$ is a reduction at $p$. The remaining assertions now follow from Lemma \ref{red=>loewy} and Theorem \ref{loewy-reg}.
\end{proof}

\section{Complexity of Fibers of a Generic Projection}\label{fibersec}

\noindent This section is devoted to the study of the fibers of generic projections of threefolds. Recall the construction of a generic projection for a variety. Let $X \subseteq \P^r$ be a non-degenerate projective variety of dimension $n$. Take an  $(r-n-2)$-dimensional linear subspace $\Lambda$ in $\P^r$ disjoint with $X$. Blowing up $\P^r$ along the center $\Lambda$ and then projecting to $\P^{n+1}$,
we obtain the diagram
$$
\begin{CD}
\Bl_{\Lambda}\P^r @>q>> \P^{n+1}\\
@VpVV  \\
\P^r.
\end{CD}
$$
The variety $X$ is naturally a subvariety of $\Bl_{\Lambda}\P^r$. Restricting $q$ to $X$ yields the morphism $\pi \colon X\rightarrow \P^{n+1}$ which is called a \emph{projection of $X$ from the center $\Lambda$}. We denote the image of $\pi$ as $\overline{X}$ which is a hypersurface in $\P^{n+1}$. For $y \in \overline{X}$, denote by $X_y:=\pi^{-1}(y)$ the fiber of the projection and denote by $$\P_y^{r-n-1}:=q^{-1}(y)=\langle \Lambda, X_y \rangle \simeq \P^{r-n-1}$$ the fiber of $q$ which is an $(r-n-1)$-dimensional linear subspace in $\P^r$. Note  that scheme-theoretically $X_y=X\cap \P_y^{r-n-1}$ and therefore $X_y\subseteq \P_y^{r-n-1}$ is a closed subscheme. As there are flexibility on the choice of the projection center $\Lambda$, we expect the projection would have some good properties. We frequently say that a general projection $\pi$ satisfies a property $\sP$, which means that there exists an open set $U$ in the Grassmannian $\G(\P^{r-n-2},\P^r)$ such that for $\Lambda\in U$ the projection from $\Lambda$ satisfies the property $\sP$. Of course, such open set $U$ depends on what kind of property $\sP$ is involved, and it should be clear from the context.

\medskip

The main goal of this section is to understand the complexity of a fiber $X_y$ of a general projection of a threefold $X$, especially the regularity of $X_y$. 
We first consider the case that $X_y$ is supported in the smooth locus $X_{\reg}$ of $X$. Recall that a zero-dimensional scheme $X$ is called \emph{curvilinear} if each local ring $\mathcal{O}_{X, x}$ is isomorphic to $\C[x]/(x^k)$ for some $k \geq 1$. It is clear that a dimension zero scheme is curvilinear if and only if it admits an embedding into a smooth curve.

\begin{proposition}\label{mather}\label{mathercor}
Let $X \subseteq \P^r$ be a projective threefold and $\pi \colon X \to \overline{X} \subseteq \P^{4}$ be a general projection from an $(r-5)$-dimensional linear subspace $\Lambda \subseteq \P^r$. Then the following hold:
\begin{enumerate}[\indent$(1)$]
	\item $X_y\cap X_{\reg}$ is curvilinear for all $y \in \overline{X}$.
	\item If a fiber $X_y$ over $y \in \overline{X}$ is supported in $X_{\reg}$, then $X_y$ has length at most $4$. In particular, $X_y \subseteq \P^{r-4}_y$ is $4$-regular.
	\item There are finitely many points $y\in \overline{X}$ such that $X_y$ is supported in $X_{\reg}$ and $X_y \subseteq \P^{r-4}_y$ is not $3$-regular; in this case $X_y$ consists of collinear four reduced points.
\end{enumerate}
\end{proposition}

\begin{proof} (1) For $x\in X_{\reg}$, denote by $T_x$ the projective tangent space of $X$ at $x$ in $\P^r$. 
We need to prove that $\dim (T_x \cap \langle x, \Lambda \rangle ) \leq 1$ for every $x \in X_{\reg}$. It is enough to prove the following claim:
$$
\text{$\dim (T_x\cap \Lambda ) \leq 0$ for every $x\in X_{\reg}$.}
$$
To show the claim, we consider the diagram of flag varieties
\[
\begin{split}
\xymatrix{
& \F(\P^1, \P^{r-5}, \P^r) \ar[dl]_-{f} \ar[dr]^-{g}& \\
 \G(\P^1, \P^r) & & \G(\P^{r-5}, \P^r)
}
\end{split}
\]
with natural projections $f$ and $g$.
Let $T:=\overline{\bigcup_{x \in X_{\reg}} T_x} \subseteq \P^r$ be the variety swept out by tangent spaces, and $V_T \subseteq \G(\P^1, \P^r)$ be the closure of the set parameterizing lines in $T_x$ for $x\in X_{\reg}$. Note that $\dim V_T \leq 7$. Let
$$
W_T:=g(f^{-1}(V_T))=\overline{\{\Lambda\in \G(\P^{r-5},\P^r)\mid \dim (T_x \cap \Lambda) \geq 1~ \text{for some } x\in X_{\reg} \}} \subseteq \G(\P^{r-5}, \P^r).
$$
Since every fiber of $f$ has dimension $5(r-6)$, it follows that $\dim f^{-1}(V_T) \leq 5r-23$ so that $\dim W_T \leq 5r-23$. Thus $\codim_{\G(\P^{r-5}, \P^r)} W_T \geq 3$. This proves the claim, so finishes the proof of $(1)$.

\smallskip

\noindent (2) We first show that
there exists a dense open subset $U \subseteq \G(\P^{r-4}, \P^r)$ parametrizing $(r-4)$-dimensional linear subspaces of $\P^r$ containing a general  $(r-5)$-dimensional linear subspace of $\P^r$.
To see this, consider the diagram of  flag varieties
\[
\begin{split}
\xymatrix{
& \F(\P^{r-5}, \P^{r-4}, \P^r) \ar[dl]_-f \ar[dr]^-g& \\
\G(\P^{r-5}, \P^r) & &\G(\P^{r-4}, \P^r)
}
\end{split}
\]
where $f$ and $g$ are natural projections. Let $V \subseteq \G(\P^{r-5}, \P^r)$ be a closed subset whose complement $ \G(\P^{r-5}, \P^r) \setminus V$ consists of general $(r-5)$-dimensional linear subspaces of $\P^r$. Let 
$$
V':=\{ L \in \G(\P^{r-4}, \P^r) \mid g^{-1}(L) \subseteq f^{-1}(V) \}
$$
be the close subset of  $\G(\P^{r-4}, \P^r)$.
Following the similar arguments in the proof of Proposition \ref{regfibsing3}, we can show that $V'$ has positive codimension in $\G(\P^{r-4}, \P^r)$. Then $U:=\G(\P^{r-4}, \P^r) \setminus V'$ is the desired dense open subset of $\G(\P^{r-4}, \P^r)$.

We now consider a closed subset
$$
\begin{array}{rcl}
W&:=&\overline{\{L \in \G(\P^{r-4}, \P^r)  \mid L \cap X \text{ is a finite scheme supported in $X_{\reg}$ and length}(L \cap X) \geq 5 \}}\\
& \subseteq &\G(\P^{r-4}, \P^r).
\end{array}
$$
We shall show 
$$
U\cap W=\emptyset.
$$
To derive a contradiction, suppose that $U \cap W \neq \emptyset$.
Take an irreducible component $Y$ of $\overline{U \cap W}$, and give a reduced scheme structure on $Y$. Then the variety $Y$ parametrizes a family $\{ L_y \}$ of $(r-4)$-dimensional linear subspaces of $\P^r$ such that a general $L_y$ contains a general $(r-5)$-dimensional linear subspace of $\P^r$  disjoint from $X$ such that $L_y \cap X$ is a finite scheme supported in $X_{\reg}$ and $\length(L_y \cap X) \geq 5$. 
By (1), $L_y \cap X$ is curvilinear.
However, the (dimension+2)-secant lemma (\cite[Theorem 1]{R}) says that a general $(r-5)$-dimensional linear subspace of $\P^r$ is not contained in any $L_y$. Thus we get a contradiction, and hence, $U \cap W =\emptyset$. Now, notice that $\langle x, \Lambda \rangle \in U$ for all $x \in X$. 
It then immediately follows that every fiber of a general projection $\pi \colon X \to \overline{X} \subseteq \P^4$ supported in $X_{\reg}$ has length at most $4$ and is $4$-regular.

\smallskip

\noindent (3) Suppose that $X_y$ is supported in $X_{\reg}$. If the length of $X_y$ is at most $3$ or is not contained in a line, then $X_y \subseteq \P_y^{r-4}$ is $3$-regular by \cite[Proposition 2.1]{Kw2}. Thus if $X_y \subseteq \P_y^{r-4}$ is not $3$-regular, then $X_y$ should be collinear of length $4$. 

By \cite[Theorem A.4]{BDELU}, the $4$-secant variety $\Sec^4 X_{\reg}$ has dimension $\leq 5$. Hence it intersects a general $(r-5)$-dimensional linear subspace $\Lambda$ in dimension $\leq 0$, which proves the finiteness of the fiber $X_y$ supported in $X_{\reg}$ such that $X_y \subseteq \P_y^{r-4}$ is not $3$-regular. Note that if $l$ is a general $4$-secant line to $X$, then $l \cap X$ consists of four distinct reduced points. This implies the last assertion.
\end{proof}

We now turn to consider the case that $X_y$ meets the singular locus of $X$.
The major difficulty of controlling the complexity of $X_y$ lies in this case, and we apply Theorem \ref{p:01} to overcome this difficulty.

\begin{proposition}\label{fibercurve}
Let $X \subseteq \P^r$ be a normal projective threefold, and $\pi \colon X \to \overline{X} \subseteq \P^4$ be a general projection from an $(r-5)$-dimensional linear subspace $\Lambda \subseteq \P^r$. 
Then the following hold:
\begin{enumerate}[\indent$(1)$]
	\item $X_y$ is supported only at $p$ for all $p \in X_{\Sing}$ with $y =\pi(p)$ except finitely many points $p_1, \ldots, p_k \in X_{\Sing}$. The exceptional case only occurs when $\dim X_{\Sing} = 1$.
	\item If $p$ is one such point $p_i$, then the fiber $X_y$ over $y=\pi(p)$ can be decomposed as two disjoint subschemes $Z$ and $q$ such that $Z$ is supported at $p$ and $q$ is a reduced point supported in $X_{\reg}$. 
\end{enumerate}
\end{proposition}

\begin{proof}
Note first that if $p \in X_{\Sing}$ is an isolated singular point, then the fiber  $X_y \subseteq  \P^{r-4}_y$ over $y=\pi(p)$ is supported at the single point $p$. 
Hence, without loss of generality, we may assume that $\dim X_{\Sing}=1$. Since $\Lambda$ is general and the variety swept out by lines connecting two singular points of $X$ has dimension $\leq 3$,  it follows that for any $p \in X_{\Sing}$, the scheme $X_y\cap X_{\Sing}$ is supported only at $p$. Also since the variety swept out by lines connecting a singular point with a nonsingular point of $X$ has dimension $\leq 5$, there are at most finitely many points $p_1, \ldots, p_k \in X_{\Sing}$ such that the support of the corresponding fiber $X_y$ has more than one points. Let $p=p_i$ for some $1 \leq i \leq k$, and $X_y$ be the fiber over $y=\pi(p)$ as usual.
We claim that 
$$
\text{$X_y \cap X_{\reg}$ is contained in a line passing through $p$.}
$$
Granting the claim, by Theorem \ref{p:01} for $Z=X_{\Sing}$, we see that $\length(X_y\cap X_{\reg}) \leq 1$. Thus $X_y$ is supported at two points $p \in X_{\Sing}$ and $q \in X_{\reg}$, and $q$ is a reduced point in $X_y$. 

It only remains to prove the claim. To this end, consider a closed subset
$$
V:=\overline{\{ L \in \G(\P^2, \P^r) \mid  \text{ $\length(L \cap X_{\reg}) \geq 2$ and $L$ intersects with $X_{\Sing}$ } \} } \subseteq \G(\P^2, \P^r).
$$
If the claim is not true, then $\dim(\Lambda \cap L) \geq 1$ for some $L \in V$.
Note that $\dim V \leq 7$.
Let $U \subseteq \G(\P^1, \P^r)$ be the closed subset parametrizing lines contained in $L\in V$. Note that $\dim U \leq 9$. Consider now the diagram of the flag varieties
\[
\begin{split}
\xymatrix{
& \F(\P^1, \P^{r-5}, \P^r) \ar[dl]_-f \ar[dr]^-g& \\
\G(\P^1, \P^r)  & &\G(\P^{r-5}, \P^r),
}
\end{split}
\]
where $f$ and $g$ are natural projections. Then we obtain a closed subset
$$
W:=g(f^{-1}(U))=\overline{\{\Lambda\in \G(\P^{r-5},\P^r)\mid \dim (\Lambda\cap L) \geq 1, L \in V \}} \subseteq \G(\P^{r-5}, \P^r).
$$
Since every fiber of $f$ has dimension $5(r-6)$, it follows that $\dim f^{-1}(U) \leq 5r-21$ so that $\dim W \leq 5r-21$. Thus $W$ has codimension at least $1$ in $\G(\P^{r-5}, \P^r)$. This implies that for a general $(r-5)$-dimensional linear subspace $\Lambda \subseteq \P^r$ and for every $L \in V$, we have $\dim (\Lambda \cap L) \leq 0$. Thus the claim follows, so we finish the proof.
\end{proof}

We are now in the position to prove the following main result of this section. The theorem is the main ingredient in the proofs of Theorems \ref{main} and \ref{main2}. It turns out that only using length is not enough to bound the regularity of $X_y$ when $X_y$ meets the singular locus of $X$. This is the place where we have to use Loewy length, especially Proposition \ref{regfibsing3}. 

\begin{theorem}\label{regfiber}
Let $X \subseteq \P^r$ be a normal projective threefold, and let $\pi \colon X \to \overline{X} \subseteq \P^4$ be a general projection from an $(r-5)$-dimensional linear subspace $\Lambda$ in $ \P^r$.
\begin{enumerate}[\indent $(1)$]
	\item Suppose that $X$ has rational singularities. Then the following hold:
	\begin{enumerate}[$(a)$]
		\item The fiber $X_y \subseteq \P_y^{r-4}$ is $4$-regular for each $y \in \overline{X}$.
		\item There are only finitely many points $y \in \overline{X}$ such that $X_y \subseteq \P_y^{r-4}$ is not $3$-regular, and such a fiber $X_y$ can be decomposed into two disjoint subscheme $Z$ and $q$ such that $\reg(Z)=3$ and $q$ is a reduced point in $X_{\reg}$. 
	\end{enumerate}
	
	\item Suppose that $X$ has Cohen-Macaulay Du Bois singularities. Then the following hold:
	\begin{enumerate}[$(a)$]
		\item The fiber $X_y \subseteq \P_y^{r-4}$ is $5$-regular for each $y \in \overline{X}$.
		\item There are only finitely many points $y \in \overline{X}$ such that $X_y \subseteq \P_y^{r-4}$ is not $4$-regular, and such a fiber $X_y$ can be decomposed into two disjoint subscheme $Z$ and $q$ such that $\reg(Z)=4$ and $q$ is a reduced point in $X_{\reg}$. 
	\end{enumerate}
	\end{enumerate}
\end{theorem}

\begin{proof}
First, consider the case that the fiber $X_y$ over $y \in \overline{X}$ is supported in $X_{\reg}$.
By Proposition \ref{mather} (2) and (3), $X_y  \subseteq \P_y^{r-4}$ is $4$-regular, and there are only finitely many points $y \in \overline{X}$ such that $X_y \subseteq \P_y^{r-4}$ is not $3$-regular; in this case, $X_y$ consists of collinear four reduced points, so we can decompose it into two disjoint subscheme $Z$ and $q$ such that $Z$ consists of three points. 

Next, consider the case that the fiber $X_y$ over $y \in \overline{X}$ meets $X_{\Sing}$, i.e., $y=\pi(p)$ for some $p \in X_{\Sing}$. 
Proposition \ref{fibercurve} $(1)$ says that $X_y$ is supported only at $p$ except finitely many points $p_1, \ldots, p_k \in X_{\Sing}$ (notice that the exceptional case only occurs when $\dim X_{\Sing}=1$).
If $p \not\in \{p_1, \ldots, p_k \}$, then Proposition \ref{regfibsing3} implies that $X_y \subseteq \P^{r-4}_y$ is $3$-regular (resp. $4$-regular) when $X$ has rational singularities (resp. $X$ has Cohen-Macaulay Du Bois singularities). Suppose that $p \in \{p_1, \ldots, p_k \}$. By Proposition \ref{fibercurve} (2), $X_y$ is decomposed into two disjoint subscheme $Z$ and $q$ such that $Z$ is supported at $p$ and $q$ is a reduced point in $X_{\reg}$.
When $X$ has rational singularities (resp. $X$ has Cohen-Macaulay Du Bois singularities), Proposition \ref{regfibsing3} implies that $Z \subseteq \P^{r-4}_y$ is $3$-regular (resp. $4$-regular), so $X_y \subseteq \P_y^{r-4}$ is $4$-regular (resp. $5$-regular). 
\end{proof}

\begin{remark}\label{rmk}
 As setting in the theorem, write $\overline{X}_4:=\{y\in \overline{X}\mid X_y \subseteq \P^{r-4}_y \text{ is not $3$-regular}  \}$. In the proof, we actually show the classification of the fiber $X_y$ over $y \in \overline{X}_4$.
\begin{enumerate}[\indent $(1)$]
	\item Suppose that $X$ has rational singularities. For any $y \in \overline{X}_4$, one of the following holds:
	\begin{enumerate}[$(a)$]
		\item $X_y$ consists of collinear four reduced points in $X_{\reg}$.
		\item $X_y$ can be decomposed into two disjoint subschemes $Z$ and $q$ such that $Z$ is supported at a singular point of $X$ with $\reg(Z)=3$ ($ll(\sO_Z) = 2$) and $q$ is a reduced point in $X_{\reg}$. This case only occurs when $\dim X_{\Sing}=1$.
	\end{enumerate}
	
	\item Suppose that $X$ has Cohen-Macaulay Du Bois singularities. For any $y \in \overline{X}_4$, one of the following holds:
		\begin{enumerate}[$(a)$]
		\item $X_y$ consists of collinear four reduced points in $X_{\reg}$.
		\item $X_y$ is supported at a singular point of $X$ with $ll(\sO_{X_y}) = 3$.
		\item $X_y$ can be decomposed into two disjoint subschemes $Z$ and $q$ such that $Z$ is supported at a singular point of $X$ with $\reg(Z)=3$ or $4$ ($ll(\sO_Z) = 2$ or $3$ respectively) and $q$ is a reduced point in $X_{\reg}$. This case only occurs when $\dim X_{\Sing}=1$.
	\end{enumerate}
\end{enumerate}
Note that $\reg(X_y) = 4$ for cases $(1.a), (1.b), (2.a), (2.b),$ and $(2.c)$ with $\reg(Z)=3$ and $\reg(X_y)=5$ for case $(2.c)$ with $\reg(Z)=4$.
\end{remark}

\section{Castelnuovo-Mumford Regularity Bounds}\label{proofsec}

\noindent We prove Theorems \ref{main} and \ref{main2} in this section by using the method of generic projection, which has been used to bound the regularities in \cite{L}, \cite{Kw1, Kw, Kw2} and \cite{Ni}. The paper \cite{Kw2} gives a very nice explanation of this method, so we refer the reader to there for further information.  

Let $X \subseteq \P^r$ be a non-degenerate locally Cohen-Macaulay projective variety of dimension $n$, codimension $e$, and degree $d$. As we have introduced in the beginning of Section \ref{fibersec}, 
let
$$
\pi \colon X \to \overline{X} \subseteq \P^{n+1}
$$ 
be a general projection from an $(r-n-2)$-dimensional linear subspace $\Lambda\subseteq \P^r$. Denote by $V=H^0(\Lambda,\sO_{\Lambda}(1))$ the vector space of linear forms on $\Lambda$. 
Let
$$
\mathcal{V}=\{V_1,V_2,\ldots, V_k\}
$$
be a collection of linear subspaces $V_j \subseteq H^0(\P^r, \mathcal{O}_{\P^r}(j))$ such that 
the natural maps 
$$
V_1 \to V,~~ V_2 \to S^2(V)=H^0(\Lambda, \sO_{\Lambda}(2))
$$ 
are isomorphisms and 
$$
V_j \to S^j(V)=H^0(\Lambda, \sO_{\Lambda}(j))
$$
are injective for all $3 \leq j \leq k$.
Then we have a map
$$
w_{\mathcal{V}} \colon V_k \otimes \mathcal{O}_{\P^{n+1}}(-k) \oplus \cdots \oplus V_1 \otimes \mathcal{O}_{\P^{n+1}}(-1) \oplus \mathcal{O}_{\P^{n+1}} \longrightarrow \pi_* \mathcal{O}_X.
$$
Let $p \colon \Bl_{\Lambda}\P^r \to \P^r$ be the blow-up of $\P^r$ along the center $\Lambda$, and $q \colon \Bl_{\Lambda}\P^r \to \P^{n+1}$ be the natural projection. Then the map $w_{\mathcal{V}}$ is induced by the natural map
$$
q_*p^* \mathcal{O}_{\P^r}(k) \longrightarrow q_*p^*\mathcal{O}_X(k)=\pi_*\mathcal{O}_X(k).
$$

For a point $y\in \overline{X}$, as discussed in Section \ref{fibersec}, the fiber $X_y$ is a subscheme of $\P^{r-n-1}_y$. Furthermore, there is a linear form $l_y\in H^0(\P^{n+1},\sO_{\P^{n+1}}(1))$ which generates $1$ in the residue field $\C(y)$ of $y$.
Note that $l_y$ is regarded as a linear form on $\P^{r-n-1}_y$ and $l^j_y$ represents the  subspace of $H^0(\P^{r-n-1}_y,\sO_{\P^{r-n-1}_y}(j))$ with the basis $l^j_y$. Furthermore, $V_j$ also can be regarded as a subspace of $H^0(\P^{r-n-1}_y, \mathcal{O}_{\P^{r-n-1}_y}(j))$ in an obvious way.
Therefore, by base change, the map $w_{\mathcal{V}}$ is surjective if and only if for all $y\in \overline{X}$ the map
$$
w_{\mathcal{V},y}:V_k \oplus (V_{k-1} \otimes l_y) \oplus \cdots \oplus (V_1 \otimes  l_y^{k-1} ) \oplus  l_y^k \longrightarrow H^0(X_y,\mathcal{O}_{X_y}(k))
$$
is surjective. Note that the vector space in the left side is a subspace of $H^0(\P^{r-n-1}_y,\sO_{\P^{r-n-1}_y}(k))$. Therefore one can check the surjectivity of $w_{\mathcal{V},y}$ for the subscheme $X_y$ in the space $\P^{r-n-1}_y$.

Suppose that $w_{\mathcal{V}}$ is surjective. Since $\pi_*\mathcal{O}_X$ is Cohen-Macaulay, it follows that
$$
E:=\ker(w_{\mathcal{V}})
$$
is a vector bundle on $\P^{n+1}$.
Assume now that the Kodaira vanishing theorem holds on $X$ (e.g., $X$ has rational or Du Bois singularities).
By \cite[Lemma 3.1]{Kw2}, we have $\reg(E^*) \leq -2$ and $\reg(E^*) \leq -3$ if and only if $X \subseteq \P^r$ is linearly normal and $H^0(\P^r, \mathcal{I}_{X|\P^r}(2))=H^1(X, \mathcal{O}_X)=0$. Then by the standard cohomological argument, the regularity of $X$ can be bounded as follows.

\begin{lemma}[{\cite[Lemma 3.2)]{Kw2}}]\label{regbound}
Under the same notations as above, we have
$$
\reg(X) \leq d-e+1 + \sum_{j=3}^k (j-2) \dim V_j.
$$
If moreover $\reg(E^*) \leq -3$, then we have
$$
\reg(X) \leq d-e+1 -\dim V_1 - \dim V_2 +  \sum_{j=4}^k (j-3) \dim V_j.
$$
\end{lemma}

We are ready to prove Theorem \ref{main}.

\begin{proof}[Proof of Theorem \ref{main}]
Take a general projection $\pi \colon X \to \overline{X} \subseteq \P^4$ from an  $(r-5)$-dimensional linear subspace $\Lambda \subseteq \P^r$, and set $V_1 \subseteq H^0(\P^r, \mathcal{O}_{\P^r}(1))$ and $V_2 \subseteq H^0(\P^r, \mathcal{O}_{\P^r}(2))$ to be any subspaces such that the natural maps $V_1 \to V=H^0(\Lambda,\sO_{\Lambda}(1))$ and $V_2 \to S^2(V)$ are isomorphisms.
For $y \in \overline{X}$, denote by $X_y:=\pi^{-1}(y)$ the fiber of the projection $\pi$ and $\P^{r-4}_y:=\langle \Lambda, X_y \rangle \simeq \P^{r-4}$.
Let $\overline{X}_4 := \{ y \in \overline{X} \mid X_y \subseteq  \P_y^{r-4} \text{ is not $3$-regular}\}$.
By definition, for every $y \in \overline{X} \setminus \overline{X}_4$, the fiber $X_y \subseteq  \P_y^{r-4}$ is $3$-regular, i.e., it is $2$-normal. Thus the  map
\begin{equation}\label{eq:02}
V_2  \otimes \mathcal{O}_{\P^{4}}(-2)  \oplus V_1  \otimes \mathcal{O}_{\P^{4}}(-1) \oplus \mathcal{O}_{\P^{4}} \longrightarrow \pi_* \mathcal{O}_X
\end{equation}
is surjective over all $y \in \overline{X} \setminus \overline{X}_4$ (i.e., tensoring with $\C(y)$ the resulting map is surjective). 

By Theorem \ref{regfiber}, $\overline{X}_4$ is a finite set, and for all $y \in \overline{X}_4$, the zero dimensional scheme $X_y$ can be decomposed into two disjoint subschemes $Z$ and $q$ such that $q$ is a reduced point in $X_{\reg}$ and $Z \subseteq \P_y^{r-4}$ is $3$-regular.
Let $H$ be a general linear form on $\P^r$ which is also a linear form on both $\P_y^{r-4}$ and $\Lambda$. 
By the choice of $H$, it gives a nonzero constant in the residue field $\C(q)$.  Let $l_y$ be a linear form on $\P^{4}$ that does not vanish at $y$. Then we have a diagram
	$$\begin{CD}
0@>>>V_2 \oplus (V_1\otimes l_y)\oplus l^2_y@>\cdot l_y>> H^3\oplus (V_2\otimes l_y) \oplus (V_1\otimes l_y^2)\oplus l^3_y @>>> H^3 @>>>0\\
@.@V\theta_2VV @V\theta_3VV @V\theta_HVV\\
0@>>>H^0(Z,\sO_Z(2)) @>\cdot l_y>> H^0(X_y,\sO_{X_y}(3))@>>> k(q) @>>> 0
\end{CD}$$\\
where $H^3$ represents a  subspace of $H^0(\P^{r-4}_y,\sO_{\P^{r-4}_y}(3))$ with the basis $H^3$ and the same for the powers of $l_y$  regarded as a linear form on $\P^{r-4}_y$. Since $Z \subseteq \P^{r-4}_y$ is $3$-regular and $H^0(\P^{r-4}_y,\sO_{\P^{r-4}_y}(2))=V_2 \oplus (V_1\otimes l_y)\oplus l^2_y$, it follows that $\theta_2$ is surjective. By the choice of $H$, the map $\theta_H$ is also surjective. Hence $\theta_3$ is surjective as well. Thus if we set $V_3:=\langle H^3 \rangle$ and  $\mathcal{V}:=\{V_1, V_2, V_3 \}$, then the map
$$
w_{\mathcal{V}} \colon V_3 \otimes \mathcal{O}_{\P^4}(-3) \oplus V_2  \otimes \mathcal{O}_{\P^{4}}(-2)  \oplus V_1  \otimes \mathcal{O}_{\P^{4}}(-1) \oplus \mathcal{O}_{\P^{4}} \longrightarrow \pi_* \mathcal{O}_X
$$
is surjective over all $y \in \overline{X}_4$. Combining the surjectivity in (\ref{eq:02}), we conclude that the map $w_{\mathcal{V}}$
is indeed surjective. Finally, by applying Lemma \ref{regbound}, we obtain 
$$
\reg(X) \leq d-e+1+1=d-e+2,
$$
so we complete the proof.
\end{proof}

\begin{remark}\label{rem:furthercase}
Let $X \subseteq \P^r$ be a non-degenerate normal projective threefold of degree $d$ and codimension $e$. Take a general projection $\pi \colon X \to \overline{X} \subseteq \P^4$ from an $(r-5)$-dimensional linear subspace $\Lambda \subseteq \P^r$.
\begin{enumerate}[$(1)$]
\item Suppose that $X$ has isolated rational singularities and the condition (\ref{techassump}) in Introduction holds, i.e.,
$$
\dim \bigcup_{l \in \Sigma} l \leq 4, \text{ where }~\Sigma=\{\text{lines } l \subseteq \P^r \mid \length(l \cap X) \geq 4 \text{ and } l \cap X_{\Sing} = \emptyset  \}.
$$
Notice that the condition (\ref{techassump}) means that there is no $y \in \overline{X}$ such that $X_y$ consists of collinear four reduced points. Thus the case $(1.a)$ in Remark \ref{rmk} cannot occur. Moreover, since $\dim X_{\Sing} \leq 0$, the case $(1.b)$ in Remark \ref{rmk} cannot occur. Thus $X_y \subseteq \P^{r-4}_y$ is $3$-regular for all $y \in \overline{X}$. Then the map
$$
w \colon V_2  \otimes \mathcal{O}_{\P^{4}}(-2)  \oplus V_1  \otimes \mathcal{O}_{\P^{4}}(-1) \oplus \mathcal{O}_{\P^{4}} \longrightarrow \pi_* \mathcal{O}_X
$$
is surjective. By Lemma \ref{regbound}, we obtain
$$
\reg(X) \leq d-e+1.
$$
\item Suppose that $X$ has Cohen-Macaulay Du Bois singularities. As in the proof of Theorem \ref{main}, we can argue that the map
$$
w \colon V_4 \otimes \mathcal{O}_{\P^4}(-4) \oplus V_3 \otimes \mathcal{O}_{\P^4}(-3) \oplus V_2  \otimes \mathcal{O}_{\P^{4}}(-2)  \oplus V_1  \otimes \mathcal{O}_{\P^{4}}(-1) \oplus \mathcal{O}_{\P^{4}} \longrightarrow \pi_* \mathcal{O}_X
$$
is surjective, where we take $V_3, V_4$ with $\dim V_3 = {e+1 \choose 3}, \dim V_4=1$. By Lemma \ref{regbound}, we obtain
$$
\reg(X) \leq d-e+3 + {e+1 \choose 3}.
$$
On the other hand, consider a general hyperplane section $S$ of $X$, which is a non-degenerate normal projective surface with Du Bois singularities in $\P^{r-1}$. By Remark \ref{rem-surDB}, we know that $\reg(S) \leq d-e+1$. Then \cite[Lemma 2.2]{NP} (see also \cite[Lemma 2.1]{KP}) implies that 
$$
\reg(\mathcal{O}_X) \leq d-e.
$$
\end{enumerate}
\end{remark}

Finally, we prove Theorem \ref{main2}. We shall use the following proposition. It is probably known to experts, but we could not find a suitable reference. So we give a proof here.

\begin{proposition}\label{cod2=>linnorm}
Let $X\subseteq\P^r$ be a projective variety of codimension $2$. Assume that $\dim X\geq 3$. Then $X \subseteq \P^r$ is linearly normal.
\end{proposition}

\begin{proof} Note that if $X$ is contained in a hyperplane in $\P^r$, then the proposition is obvious. So in the sequel we assume $X \subseteq \P^r$ is non-degenerate. We first show that the only non-degenerate surface $S$ in $\P^4$ which is not linearly normal is obtained by an isomorphic projection of the second Veronese surface in $\P^5$. Indeed, if $S \subseteq \P^4$ is not linearly normal, then it can be projected isomorphically from $\P^5$. Thus in $\P^5$, the secant variety of $S$ should have dimension $4$. By Severi's theorem for singular surfaces \cite{D}, $S \subseteq \P^5$ is either a cone or the second Veronese surface in $\P^5$. But if $S$ is a cone, then the tangent space of $S$ at a vertex point is the whole space $\P^5$. So in this case, $S$ cannot be projected to $\P^4$ isomorphically. Hence $S $ is a projected second Veronese surface in $\P^4$.

Now let $X$ be a threefold in $\P^5$. Suppose otherwise $X \subseteq \P^5$ is not linearly normal so that it is obtained by an isomorphic projection of $X \subseteq \P^6$.
Let $Y$ be a general hyperplane section of $X \subseteq \P^5$. Then $Y \subseteq \P^4$ is not linearly normal. So by the above argument, $Y$ is a projected second Veronese surface in $\P^4$. This implies that $X \subseteq \P^6$ has degree $4$, and therefore, it is a variety of minimal degree. Hence $X \subseteq \P^6$ is either a rational normal scroll or a cone over the second Veronese surface. If $X$ is a cone, then the tangent space of $X$ at a vertex point is the whole $\P^6$ so that $X$ cannot be projected isomorphically to $\P^5$. Hence $X$ is not a cone, and it is a rational normal scroll. However in this case the tangent variety of $X$ fills all $\P^6$, so again it cannot be projected isomorphically into $\P^5$. Hence $X \subseteq \P^5$ has to be linearly normal.

Finally, consider a non-degenerate projective variety $X \subseteq \P^r$ of codimension $2$ and dimension $\geq 4$. By taking general hyperplane section which is linearly normal by induction, we see that $X \subseteq \P^r$ is linearly normal as claimed.
\end{proof}

\begin{proof}[Proof of Theorem \ref{main2}]
Suppose first that $X \subseteq \P^5$ is contained in a hyperquadric, i.e., $H^0(\P^5, \mathcal{I}_{X|\P^5}(2)) \neq 0$. If $X \subseteq \P^5$ is a complete intersection, then $\reg(X) \leq d-1$ and the equality holds exactly when $X \subseteq \P^5$ is of type $(2,2)$.  If $X \subseteq \P^5$ is not a complete intersection, then $\reg(X)=\frac{d+1}{2}$ by \cite[Proposition 2.9 (2)]{Kw}. Thus $\reg(X) \leq d-1$ and the equality holds when $d=3$, i.e., $X \subseteq \P^5$ is a variety of minimal degree. From now on, we assume that $X \subseteq \P^5$ is not contained in a hyperquadric, i.e., $H^0(\P^5, \mathcal{I}_{X|\P^5}(2))=0$. We actually prove that $\reg(X) \leq d-2$, which completes the proof.

As in the proof of Theorem \ref{main}, take a general projection $\pi \colon X \to \overline{X} \subseteq \P^4$ from a general point in $\P^5$. By Theorem \ref{regfiber}, $\pi^{-1}(y)=X_y \subseteq \P_y^1$ is $5$-regular. Thus the map
$$
w \colon \mathcal{O}_{\P^4}(-4) \oplus \mathcal{O}_{\P^4}(-3) \oplus \mathcal{O}_{\P^4}(-2) \oplus  \mathcal{O}_{\P^4}(-1) \oplus \mathcal{O}_{\P^4} \longrightarrow \pi_* \mathcal{O}_X
$$
is surjective, and the kernel of $w$ is a vector bundle $E$ on $\P^4$. We now claim that 
$$
\reg(E^*) \leq -3.
$$ 
By \cite[Lemma 3.1]{Kw2}, it is sufficient to show that $X \subseteq \P^5$ is linearly normal and $H^1(X, \mathcal{O}_X)=0$. By Proposition \ref{cod2=>linnorm}, $X \subseteq \P^5$ is linearly normal. On the other hand, \cite[Corollary 5.3]{FL} says that $X$ is simply connected and therefore $H^1(X,\C)=0$. 
Since $X$ has Du Bois singularities, it follows that $H^1(X, \C) \to H^1(X, \mathcal{O}_X)$ is surjective. Thus we obtain $H^1(X, \mathcal{O}_X)=0$. We have shown the claim. Now, by applying  Lemma \ref{regbound}, we see that $\reg(X) \leq d-2$. 
\end{proof}

$ $

\end{document}